\documentclass[12pt]{amsart}

\usepackage{amsfonts,latexsym,rawfonts,amsmath,amssymb,amsthm,cite,color}
\usepackage[plainpages=false]{hyperref}
\usepackage{graphicx}

\numberwithin{equation}{section}
\RequirePackage{color}
\theoremstyle{plain}
  \newtheorem{theorem}{Theorem}[section]
  \newtheorem{lemma}{Lemma}[section]

\theoremstyle{definition}
  \newtheorem{definition}{Definition}[section]

\theoremstyle{remark}

\newcommand{\beq}{\begin{equation}}
\newcommand{\eeq}{\end{equation}}
\newcommand{\beqs}{\begin{eqnarray*}}
\newcommand{\eeqs}{\end{eqnarray*}}
\newcommand{\beqn}{\begin{eqnarray}}
\newcommand{\eeqn}{\end{eqnarray}}
\newcommand{\beqa}{\begin{array}}
\newcommand{\eeqa}{\end{array}}

\def\R{\mathbb R}

\begin{document}
\title{the hadamard variational formula for riesz capacity and its applications}

\author{Lu Zhang}
\address{School of Mathematics, Hunan University, Changsha, 410082, China}
\email{luzhang@hnu.edu.cn}


\maketitle

\maketitle

\baselineskip18pt

\parskip3pt
\begin{abstract}

In this paper, our focus lies on a fundamental geometric invariant known as Riesz capacity, which holds an essential position in potential theory. We establish the Hadamard variational formula for Riesz capacity of convex bodies. As a meaningful application, we derive a Serrin-type symmetry result for an overdetermined problem.

\end{abstract}


\section{Introduction}
The classical Minkowski problem is a core area in convex geometric analysis, which is  to find a convex body whose surface area measure equals a given spherical Borel measure.
Minkowski \cite{MH1897,MH1903} was the first to pose and solve the Minkowski problem when the given measure is either discrete or has a continuous density. Aleksandrov \cite{A1938} and Fenchel-Jessen \cite{FJ1938} independently solved the problem for arbitrary measures in the 1930s.
In particular, Aleksandrov  \cite{A1938} proved that surface area measure is the differential of volume and considered  the existence of minimizers of a related minimization problem to solve the problem.
Taking the place of volume, Jerison \cite{JD1996} considered another fundamental geometric invariant, electrostatic capacity. He established the Hadamard variational formula for capacity and then solved the capacitary Minkowski problem through a analogous related minimization problem. Subsequently, Colesanti et.al\cite{CNSXYZ2015} nontrivially extended Jerison's work on electrostatic capacity to $p$-capacity.
From the above works, it can be seen that the Hadamard variational formula plays a crucial role in solving the Minkowski type problem.

 In this paper, we are concerned with another significant geometric invariant, Riesz capacity, which plays a crucial role in potential theory in \cite{L1972}.

Let $\Omega\subset \R^{n}$ be a compact set and $\alpha\in (0,n)$, the $\alpha$-Riesz potential energy of $\Omega$ is defined by
\begin{equation*}
  I_{\alpha}(\Omega)=\text{inf} \Big \{\int_{\R^{n}}\int_{\R^{n}}\frac{1}{|x-y|^{n-\alpha}}d\mu(x)d\mu(y): \text{supp}(\mu)=\Omega, \mu(\Omega)=1 \Big \}.
\end{equation*}
 The infimum in the definition of $I_{\alpha}(\Omega)$ can be achieved by a Radon measure $\mu_{\Omega}$ in \cite{L1972}.
The Riesz potential of $\Omega$ is defined as
\begin{equation*}
  V(x)=\int_{\R^{n}}\frac{1}{|x-y|^{n-\alpha}} d\mu_{\Omega}(y),
\end{equation*}
and then
\begin{equation*}
  I_{\alpha}(\Omega)=\int_{\R^{n}} V(x)d\mu_{\Omega}(x).
\end{equation*}
It can be verified from \cite{GNB2015} that the Riesz potential $V$ satisfies
\begin{equation*}
  (-\Delta)^{\alpha/2}V=c_{n,\alpha}\mu_{\Omega},
\end{equation*}
in the distributional sense, where $(-\Delta)^{s}$ stands for the $s$-fractional Laplacian and
$c_{n,\alpha}$ is a positive constant depending only on $n$ and $\alpha$. Furthermore, according to \cite{NR2015}, the function $\frac{V}{I_{\alpha}(\Omega)}$ solves the following equations
\begin{equation*}
\left\{
\begin{aligned}
& (-\Delta)^{\alpha/2}u=0, \ \ \ \ \ \ \ \ \ \ \ in \ \ \R^{n}\backslash \Omega,  \\
&  u=1  , \ \ \ \ \ \ \ \ \ \ \ \ \ \ \ \ \ \ \
 \  \ \ in \ \ \Omega,  \\
&  \lim\limits_{|x|\rightarrow \infty}u(x)=0.
\end{aligned}
\right.
\end{equation*}

From \cite{L1972}, the $\alpha$-Riesz capacity is defined by
\begin{equation*}
  \text{Cap}_{\alpha}(\Omega)=\frac{c_{n,\alpha}}{I_{\alpha}(\Omega)}.
\end{equation*}
It is important to note that the definition of $\alpha$-Riesz capacity is not unique.  When $\alpha \in(0,2)$, there exists an equivalent definition (see section \ref{sec2}):
\begin{equation}\label{a1}
\text{Cap}_{\alpha}(\Omega)=\text{inf}\{\|u\|_{\mathring{H}^{\alpha/2}(\R^{n})}^{2}: u\in C_{c}^{\infty}(\R^n),u\geq \chi_{\Omega}\},
\end{equation}
where $\chi_{\Omega}$ is the characteristic function of the set $\Omega$ and
\begin{equation*}
\|u\|_{\mathring{H}^{\alpha/2}(\R^{n})}^{2}=\int_{\R^{n}}\int_{\R^{n}}\frac{|(u(x)-u(y))|^{2}}{|x-y|^{n+\alpha}}dxdy.
\end{equation*}

The connection with the classical electrostatic(Newtonian) capacity, defined as
\begin{equation*}
\text{Cap}_{2}(\Omega)=\text{inf}\{\|\nabla u\|_{L^{2}(\R^{n})}^{2}: u\in C_{c}^{\infty}(\R^n),u\geq \chi_{\Omega}\},
\end{equation*}
is established by the equality
\begin{equation*}
\text{Cap}_{1}(\Omega)=\text{Cap}_{2}(\Omega\times \{0\}),
\end{equation*}
for any compact set $\Omega\subset \R^{n}$,  for example see \cite{QZ2023}.

There have been few works  on the shape optimization problems involving the fractional Laplacian.
Dalibard and G\'{e}rard-Varet \cite{DG2013} considered the energy $J(\Omega)$ associated with the solution $u_{\Omega}$ of basic Dirichlet problem $(-\Delta)^{1/2}u=1$ in $\Omega$, $u=0$ in $\R^{n}\setminus \Omega$. They obtained the Hadamard variational formula for $J(\Omega)$ with a bounded open set $\Omega\subset \R^{2}$ of class $C^{\infty}$. Subsequently, Djitte, Fall and Weth \cite{DF2021} derived a shape derivative formula for the family of principle Dirichlet eigenvalues $\lambda_{s,p}(\Omega)$ with a bounded open set $\Omega \subset \R^{n}$ of class $C^{1,1}$. Recently,
Muratov, Novaga and Ruffini \cite{MNR2022} computed the variational formula for 1-Riesz capacity with a $C^{2}$ bounded set $\Omega\subset \R^{2}$.

To obtain the variational formula for Riesz capacity, we compute the first variation of an auxiliary functional
and establish its correlation with Riesz capacity.
 By leveraging some estimates for crucial maps as discussed in  \cite{DF2021}, we are able to address the higher-dimensional case through meticulous calculations.
 Subsequently, utilizing the connection between the auxiliary functional and Riesz capacity, we derive the Hadamard variational formula
for $\alpha$-Riesz capacity of convex bodies in $\R^{n}$ under smooth conditions
for $\alpha\in (0,2)$.
The variational formula for Riesz capacity is established in the following theorem.
\begin{theorem}\label{MT1}
Let $\Omega$ and $L$ be two convex bodies of class $C^{2}$ with support functions $h_{\Omega}$ and $h_{L}$ respectively.
Let $\Omega_{t}$ be the Wulff shape of  $h_{t}=h_{\Omega}+th_{L}$ and $\nu$ be the Gauss map of  $\Omega$. Suppose $u_{\Omega}$ is the minimizer of $Cap_{\alpha}(\Omega)$ in \eqref{a1}, then we have
\begin{equation*}
\frac{d}{dt} \mathrm{Cap}_{\alpha}(\Omega_{t})\Big{|}_{t=0} =c_{0}\int_{\partial \Omega}|\partial _{\nu}^{\alpha/2} u_{\Omega}(\sigma)|^{2}h_{L}(\nu_{\Omega}(\sigma))d\mathcal{H}^{n-1}(\sigma),
\end{equation*}
where $c_{0}=c_{\alpha/2}a_{n,\alpha/2}c_{n,\alpha/2}$ is a positive constant and $\partial _{\nu}^{\alpha/2}u_{\Omega} $ plays the role of an fractional normal derivative.
\end{theorem}

In the field of convex geometry,
it is natural to define $\alpha$-Riesz capacitary measure $\mu_{Cap_{\alpha}}(\Omega,\cdot)$ by
\begin{equation*}
\mu_{Cap_{\alpha}}(\Omega, \eta)=\int_{\nu^{-1}(\eta)}|\partial_{\nu}^{\alpha/2}u(x)|^{2}d\mathcal{H}^{n-1}(x),\ \ \ \ \text{for}\ \text{any}\ \text{Borel}\ \text{set}\ \eta\subset S^{n-1},
\end{equation*}
and pose the Minkowski problem for $\alpha$-Riesz capacity:
 let $\mu$ be a finite Borel measure on the unit sphere $S^{n-1}$, what are necessary and sufficient conditions for the existence of a convex body $\Omega$ satisfying
\begin{equation*}
  \mu=\mu_{Cap_{\alpha}}(\Omega,\cdot)?
\end{equation*}
The Minkowski problem for $\alpha$-Riesz capacitary measure is interesting. However,
we cannot solve the Minkowski problem like \cite{JD1996},
due to a strong dependency on the weak continuity
of the measure $\mu_{Cap_{\alpha}}(\Omega, \cdot)$.


As a meaningful application of Theorem \ref{MT1},  we deduce an overdetermined problem for Riesz capacity and then prove Serrin type symmetry result for our overdetermined problems.

The study of overdetermined problems has a long history. In the 1970s, Serrin\cite{S1971} researched the following overdetermined problem
\begin{equation*}
\left\{
\begin{aligned}
&\Delta u +1=0\ \ \ \ \ \ \ \ \ \ \ \  \ \  \ \ \ in \ \ \Omega,\\
&u=0  \ \ \ \ \ \ \ \ \ \ \ \ \ \ \ \ \ \ \ \ \  \ \ \ \ on  \ \partial \Omega,\\
& \frac{\partial u}{\partial \nu}=const\ \ \ \ \ \ \ \ \ \ \ \ \ \ \ \ \ on\ \partial\Omega.
\end{aligned}
\right.
\end{equation*}
 The author used the way of moving parallel planes and proved that the domain  is a ball when $\Omega$ is a $ C^{2}$ domain. The Serrin-type problem on the interior domain has a lot results, for example, see \cite{W1971,CH1998,PW1989,BNS2008,S2012}.
Additionally, extensions of Serrin's result have also been given to other domains such as exterior domains (see \cite{AB1998,GS1999,MR2000,R1997,S2001}) and annular domains (see \cite{A1992,S2001,WGS1994,P1990,R1995,P1989}).

The generalization of Serrin's result to fractional Laplacian operator has been studied.
Fall and Jarohs \cite{FJ2015} considered the following fractional Laplacian equation
\begin{equation*}
\left\{
\begin{aligned}
&(-\Delta)^{s} u=f(u)\ \ \ \  \ \ in \ \ \Omega,\\
&u=0  \ \ \ \ \ \ \ \  \ \ \ \ \ \ \ \ \ \ \ in  \ \ \R^{n}\backslash \Omega,\\
& \partial_{\nu}^{s}u=const\ \ \ \ \ \ \ \ \ \ on\ \partial \Omega,\\
\end{aligned}
\right.
\end{equation*}
and obtained the radial symmetry of domain for locally Lipschitz function $f$.
In particular, the case for $n=2, s=1/2$ and $f=1$ was solved in \cite{DG2013}. Afterwards,
Soave and Valdinoci\cite{SV2019} extended the Serrin type problem to the fractional Laplacian equation on an exterior set of a multiply connected domain
and got that the domain is a ball and the solution is radially symmetric.
 Some related results, we refer to Li and Li\cite{LL2016}.


Besides that,
Fragal\`a \cite{F2012} generalized the Serrin's result to the Dirichlet problem with an additional
Neumann boundary condition. He used the Hadamard variational formula for functionals associated
equations and the Brunn-Minkowki inequality to obtained the symmetry result.
Inspired by her work, there were other analogous results. Dai and Shi\cite{DS2013}
deduced the Hadamard variational formula for Hessian eigenvalue function
and an overdetermined problem, and then obtained the symmetric result for their
overdetermined problem. Later, Huang, Song and Xu\cite{HSX2018} got the Hadamard
variational formula for $p$-torsion and $p$-eigenvalue and then deduced and
proved  the symmetric result for related overdetermined problem. Recently,
the overdetermined problem for $p$-capacity was also solved by Ji\cite{J2020}.

Our next result provides a characterization of constrained local minimum of $\text{Cap}_{\alpha}$ (see section \ref{sec4} for related definitions).
\begin{theorem}\label{A8}
Let $\Omega$ be a  bounded convex domain that contains the origin in its interior and $\partial\Omega \in C^{2}$. If $\Omega$ is a volume constrained local minimum for $\Omega\rightarrow \text{Cap}_{\alpha}(\Omega)$, then $\Omega$ is a ball.
\end{theorem}

Furthermore, we also study the following  overdetermined boundary value problem and
obtain the symmetry result by the variational formula for $\alpha$-Riesz capacity.
\begin{theorem}\label{MT3}
Let $\Omega$ be a bounded convex domain of class $C^{2}$ that contains the origin in its interior and $G(x)$ be
the Gauss curvature of $\partial \Omega$. If there exists a solution $u$  to the following overdetermined boundary value problem
\begin{equation*}
\left\{
\begin{aligned}
&(-\Delta)^{\frac{1}{2}}u=0\ \ \ \ \ \ \ \ \ \ \ \ \ \ \ \ \ \  \ \ in \ \ \ \R^{n}\setminus \bar{\Omega},\\
&u=1  \ \ \ \ \ \ \ \ \ \ \ \ \ \ \ \ \ \ \ \ \ \ \ \ \ \ \ \  \ on  \ \partial \Omega,\\
&\lim\limits_{|x|\rightarrow \infty}u(x)=0,\\
& |\partial_{\nu}^{\frac{1}{2}}u(x)|^{2}=cG(x)\ \ \ \ \ \ \ \ \ \ \ \ \ \ on\ \partial \Omega,
\end{aligned}
\right.
\end{equation*}
then $\Omega$ is a ball.
\end{theorem}

This article is organized as follows. In section \ref{sec2}, we provide some basic notations
related to convex bodies and introduce fundamental results concerning the
fractional Sobolev space, Riesz potential energy and Riesz capacity. In
section \ref{sec3}, we establish
 the Hadamard variational formula for $\alpha$-Riesz capacity in the smooth case for $n\geq2$.
In section \ref{sec4}, we utilize  the Hadamard variational formula to derive the symmetric results for the overdetermined boundary value problem.
\section{Preliminaries}
\label{sec2}

\subsection{Convex geometry}
In this subsection, we will provide some basics and notations about convex bodies. For more details,  Schneider's book \cite{S2014} is recommended for reference.

Let $\R^{n}$ be the $n$-dimensional Euclidean space. The standard inner product of the vectors $x,y \in \R^{n}$ is denoted as $x\cdot y$. The Lebesgue measure is denoted as $|\cdot|$. The volume of the unit sphere in $\R^{n}$ is denoted as $\omega_{n}$.

A compact convex set with non-empty interior is called  convex body. The set of all convex bodies in $\R^{n}$
will be written as $\mathcal{K}^{n}$. The set  $\mathcal{K}_{o}^{n}$ represents   convex bodies that contain the origin in their interiors.
The space $C^{2}(\R^{n})$ stands for all twice differentiable functions and the space $C_{c}^{\infty}(\R^{n})$ stands for all infinitely differentiable functions with compact supports in $\R^{n}$.

 Let $K$ and $L$ be two compact convex sets, the $Hausdorff$ $metric$ between $K$ and $L$ is defined as
\begin{equation*}
  d(K,L)=\mbox{min}\{t>0: K\in L+tB, L\in K+tB  \}.
\end{equation*}
 A sequence of compact convex sets $K_{i}$ converges to $K$ if  $d(K_{i},K)\rightarrow 0$, as $ i \rightarrow\infty$.

Let $K,L \in \mathcal{K}^{n}$ and $a,b\in \R^{+}$, the $Minkowski$ $combination$ is defined by
\begin{equation*}
  aK+bL=\{ax+by: x\in K, y\in L\}.
\end{equation*}

Let $K$ be a compact convex subset of $\R^{n}$, the $support function$ $h_{K}: \R^{n}\rightarrow \R$ of $K$ is defined by
\begin{equation*}
  h_{K}(x)=\mbox{max} \{x\cdot y: y\in K\}.
\end{equation*}
The support function $h_{K}$ is convex and homogeneous of degree 1.
Moreover, if $K\in \mathcal {K}_{o}^{n}$, then $h_{K}$ is positive and continuous.

Let $f$ be a positive continuous function on $S^{n-1}$, the $Wulff$ $shape$ $[f]$ is written as
\begin{equation*}
  [f]=\{x\in \R^{n}: x\cdot v \leq f(v),\ \forall\ v\in S^{n-1}\}.
\end{equation*}
 It is apparent that $h_{[f]} \leq f$ and $[h_{K}]=K$ for each $K\in \mathcal{K}_{o}^{n}$.

 Let $K\in \mathcal{K}_{o}^{n}$ and $x\in \R^{n}\setminus\{0\}$, the $radial$ $function$ $\rho_{K}(x): \R^{n}\rightarrow \R$  of $K$ is defined by
\begin{equation*}
  \rho_{K}(x)=\mbox{max}\{t\geq 0: tx \in K\}.
\end{equation*}
The radial function is homogeneous of degree -1. Then $\partial K=\{\rho_{K}(u)u: u\in S^{n-1}\}$.


\subsection{The fractional Sobolev space}
This subsection is devoted to the definition of the fractional Sobolev spaces.

Let $s\in (0,1)$, the fractional Sobolev space $H^{s}(\R^{n})$ is defined as follows
\[
H^{s}(\R^{n}):=\left\{ u\in L^2(\R^{n}):\ \frac{|u(x)-u(y)|^2}{|x-y|^{n+2s}} \in L^2(\R^{n} \times \R^{n})   \right\}.
\]
It is an Hilbert space, endowed with the natural norm
\[
\|u\|_{H^{s}(\R^{n})}^{2}:=  \int_{\R^{n}} |u|^2 dx+ \int_{\R^{n}}\int_{\R^{n}} \frac{|u(x)-u(y)|^2}{|x-y|^{n+2s}} dxdy ,
\]
the term
\[
\|u\|_{\mathring{H}^{s}(\R^{n})}^{2}:=\int_{\R^{n}}\int_{\R^{n}} \frac{|u(x)-u(y)|^2}{|x-y|^{n+2s}} dxdy
\]
is so-called Gagliardo (semi) norm of $u$.
The dual space to $\mathring{H}^{s}(\R^{n})$ is denoted $\mathring{H}^{-s}(\R^{n})$.

Let $\mathcal{S}$ be the Schwartz space of rapidly decreasing $C^{\infty}$ functions in $\R^n$. For any function $u\in \mathcal{S}$, the $s$-fractional Laplacian $(-\Delta)^{s}$ is defined by
\begin{equation}\label{F-S-1}
(-\Delta)^{s}u(x)=c_{n,s}P.V.\int_{\R^{n}}\frac{u(x)-u(y)}{|x-y|^{n+2s}}dy,
\end{equation}
where $P.V.$ stands for the Cauchy principle value and $c_{n,s}$ is a constant depending only on $n$ and $s$. In addition, this operator $(-\Delta)^{s}$ has alternate form
\[
(-\Delta)^{s}u(x)=\frac{1}{2}c_{n,s}\int_{\R^n} \frac{2u(x)-u(x+y)-u(x-y)}{|y|^{n+2s}}dy.
\]

Note that the operator $(-\Delta)^{s}u(x)$ is well-defined when $u\in C^{\beta+2s}(\R^{n})\cap L^{\infty}(\R^{n})$ for some $\beta\in (0,1)$. In fact, when $u\in L^{\infty}(\R^{n})$, then
\begin{equation*}
\begin{aligned}
&\Big{|}\int_{|y|>1} \frac{2u(x)-u(x+y)-u(x-y)}{|y|^{n+2s}}dy\Big{|}\leq 4\|u\|_{L^{\infty}(\R^{n})}\int_{|y|>1} \frac{1}{|y|^{n+2s}}dy<\infty.
\end{aligned}
\end{equation*}
When $u\in C^{\beta+2s}(\R^{n})$, then $|2u(x)-u(x+y)-u(x-y)|\leq C |y|^{\beta+2s}$ and
\begin{equation*}
\begin{aligned}
&\Big{|}\int_{|y|\leq 1} \frac{2u(x)-u(x+y)-u(x-y)}{|y|^{n+2s}}dy\Big{|}\leq C\int_{|y|\leq 1} \frac{1}{|y|^{n-\beta}}dy<\infty.
\end{aligned}
\end{equation*}

Moreover, note that in \cite {G} if $u\in C_c^{\infty}(\R^n)$, then $(-\Delta)^{s}u \in C^{\infty}(\R^n)$, but is not compactly supported and
\[
(-\Delta)^{s}u=O(|x|^{-(n+2s)}),\ \ \ \mathrm{as}\ |x|\rightarrow \infty.
\]


The basic symmetry of the operator $(-\Delta)^{s}$ in \cite {G} is that
\begin{equation}\label{aa7}
\int_{\R^{n}}u(x)(-\Delta)^{s}v(x)dx=\int_{\R^{n}}v(x)(-\Delta)^{s}u(x)dx,
\ \ \text{for}\ \text{any}\ u,v\in \mathcal{S}.
\end{equation}
It is also a fact in \cite{DPV2012} that
\begin{equation}\label{aa31}
\int_{\R^n} |(-\Delta)^{s/2}u(x)|^2 dx= \int_{\R^n} \int_{\R^n} \frac{|u(x)-u(y)|^2}{|x-y|^{n+2s}}dxdy.
\end{equation}

\subsection{The Riesz potential energy and Riesz capacity}
In this subsection, we list some elementary results about the Riesz potential energy and Riesz capacity.

For quick reference, we repeat the definition of $\alpha$-Riesz potential energy of a set $\Omega$,
\begin{equation*} I_{\alpha}(\Omega)=\text{inf} \Big \{\int_{\R^{n}}\int_{\R^{n}}\frac{1}{|x-y|^{n-\alpha}}d\mu(x)d\mu(y): \text{supp}(\mu)=\Omega, \mu(\Omega)=1 \Big \}.
\end{equation*}
 The infimum in the definition of $I_{\alpha}(\Omega)$ can be achieved by a Radon measure $\mu_{\Omega}$ when $\Omega$ is a compact set.
  The $\alpha$-Riesz capacity is defined by
\begin{equation}\label{b1}
  \text{Cap}_{\alpha}(\Omega)=\frac{c_{n,\alpha}}{I_{\alpha}(\Omega)}.
\end{equation}

 When $\alpha \in(0,2)$, there is an alternative definition of $\alpha$-Riesz capacity,
\begin{equation}\label{c4}
\text{Cap}_{\alpha}(\Omega)=\text{inf}\{\|u\|_{\mathring{H}^{\alpha/2}(\R^{n})}^{2}: u\in C_{c}^{\infty}(\R^n),u\geq \chi_{\Omega}\}.
\end{equation}

These two definitions are  equivalent. For the completeness of the content, we give the proof. It is a fact that the Riesz potential $V(x):=\int_{\R^{n}}\frac{1}{|x-y|^{n-\alpha}} d\mu_{\Omega}(y)$ satisfied
\begin{equation}\label{aa30}
  (-\Delta)^{\alpha/2}V=c_{n,\alpha}\mu_{\Omega},
\end{equation}
in the distributional sense in \cite{GNB2015}.
Since $\text{supp}(\mu_{\Omega})=\Omega$, then
  $(-\Delta)^{\alpha/2}V=0 $ in $\R^{n}\backslash \Omega$. Moreover,
 as indicated in \cite{NR2015}, the function $\frac{V}{I_{\alpha}(\Omega)}$ is the unique solution of the following equations
\begin{equation}\label{b3}
\left\{
\begin{aligned}
& (-\Delta)^{\alpha/2}u=0, \ \ \ \ \ \ \ \ \ \ \ \text{in} \ \ \R^{n}\backslash \Omega,  \\
&  u=1  , \ \ \ \ \ \ \ \ \ \ \ \ \ \ \ \ \ \ \ \ \  \ \text{in} \ \ \Omega,  \\
&  \lim\limits_{|x|\rightarrow \infty}u(x)=0.
\end{aligned}
\right.
\end{equation}
Thus, by \eqref{aa30} and \eqref{aa31}, we have
\begin{equation*}
I_{\alpha}(\Omega)=\int_{\R^{n}}V(x)d\mu_{\Omega}(x)=\frac{1}{c_{n,\alpha}}\int_{\R^{n}}V(x)(-\Delta)^{\alpha/2}Vdx=\frac{1}{c_{n,\alpha}}\|V\|_{\mathring{H}^{\frac{\alpha}{2}}(\R^{n})}^{2}.
\end{equation*}
Let $u_{\Omega}:=\frac{V}{I_{\alpha}(\Omega)}$, then arguing as in the proof of Theorem 11.16 in \cite{L2001} we have $u_{\Omega}$ is the minimizer of $\text{Cap}_{\alpha}(\Omega)$ and
\begin{equation*}
\text{Cap}_{\alpha}(\Omega)=\|u_{\Omega}\|_{\mathring{H}^{\frac{\alpha}{2}}(\R^{n})}^{2}=\frac{\|V\|_{\mathring{H}^{\frac{\alpha}{2}}(\R^{n})}^{2}}{I_{\alpha}(\Omega)^{2}}=\frac{c_{n,\alpha}}{I_{\alpha}(\Omega)}.
\end{equation*}
The proof of the equivalence of two definitions is completed.

\section{the hadamard variational formula for riesz capacity}
\label{sec3}
In this section, we establish the Hadamard variational formula for $\alpha$-Riesz capacity.  Initially, we compute the first variation of an auxiliary functional and show the relation with Riesz capacity. Obtaining the first variation of the auxiliary functional in the case of $n\geq2$ relies on careful estimates in \cite{DF2021}.
In particular, in the case of $n=2$ and $\alpha=1$, the variational formula was obtained in \cite{MNR2022}. For brevity, we write $s=\frac{\alpha}{2}$.
\subsection{The first variation of an auxiliary functional}
Let $\Omega$ be an open set of class $C^{2}$ in $\R^n$, not necessarily bounded, and $f\in  L^{\frac{2n}{n+2s}}(\R^n)$. We consider the infimum problem
\begin{equation*}
\text{inf} \{I_{\Omega,f}(v): v\in \mathring{H}^{s}(\R^n)\ \text{and} \ v|_{\Omega^{c}}=0\},
\end{equation*}
where the functional
$I_{\Omega,f}$ is defined by
\begin{equation*}
I_{\Omega,f}(v)=\left\{
\begin{aligned}
&\frac{1}{2}\|v\|_{\mathring{H}^{s}(\R^n)}^{2}-\int_{\R^n}fvdx,\ \ \text{if}\ v\in \mathring{H}^{s}(\R^n)\  \text{and} \ v|_{\Omega^{c}}=0,\\
&\infty,\ \ \ \ \ \ \ \ \ \ \ \  \ \ \ \ \ \ \ \ \ \ \ \ \ \ \ \ \   \text{otherwise}.
\end{aligned}
\right.
\end{equation*}

  Since the space $\mathring{H}^{s}(\R^n)$ is continuously embedded into $L^{\frac{2n}{n-2s}}(\R^n)$ (see theorem 6.5 of \cite{DPV2012}), then there exists a unique minimizer $u_{\Omega,f}\in \mathring{H}^{s}(\R^n)$ of $I_{\Omega,f}$ such that
 \begin{equation}\label{d1}
\left\{
\begin{aligned}
&(-\Delta)^{s}u_{\Omega,f}=f, \ \ \ \ \ \ \mathrm{in}\  \Omega,\\
 &u_{\Omega,f}=0,\ \ \ \ \ \ \ \ \ \ \ \ \ \ \ \mathrm{in}\  \Omega^{c}\\
\end{aligned}
\right.
 \end{equation}
 in the distributional sense, that is
 \begin{equation*}
 \int_{\R^n}u_{\Omega,f}(x)(-\Delta)^{s}\varphi(x)dx=\int_{\R^n}f(x)\varphi(x)dx,\ \ \ \ \text{for}\ \text{all}\ \varphi \in C_{c}^{\infty}(\R^n).
 \end{equation*}
 Moreover, when $u_{\Omega,f}|_{\Omega}\in C_{loc}^{\beta+2s}(\Omega)\cap L^{\infty}(\Omega)$ for some $\beta\in (0,1)$, then \eqref{d1} holds pointwise in $\Omega$.
Therefore, we have
 \begin{equation}\label{c3}
 J_{f}(\Omega):=\min I_{\Omega,f}=-\frac{1}{2}\int_{\Omega}u_{\Omega,f}(x)f(x)dx=-\frac{1}{2}\int_{\Omega}u_{\Omega,f}(x)(-\Delta)^{s}u_{\Omega,f}(x)dx.
 \end{equation}

In the following lemma, we prove the regularity theory for the Dirichlet problem \eqref{d1}. It shows that $u_{\Omega,f}|_{\Omega}\in C_{loc}^{\beta+2s}(\Omega)\cap L^{\infty}(\Omega)$ for some $\beta\in (0,1)$ when $f\in L^{\infty}(\R^{n})\cap L^{\frac{2n}{n+2s}}(\R^{n})$ and $f|_{\Omega}\in C_{loc}^{\beta}(\Omega)$. Thus, \eqref{d1} holds pointwise in $\Omega$.

 \begin{lemma}\label{A4}
Let $\Omega$ be an open set with boundary of class $C^{2}$ and $f\in L^{\infty}(\R^{n})\cap L^{\frac{2n}{n+2s}}(\R^{n})$. Let $u_{\Omega,f}$ be the minimizer of $I_{\Omega,f}$, then $u_{\Omega,f}\in {L^{\infty}(\R^n)}$.  Additionally, if $f|_{\Omega}\in C_{loc}^{\beta}(\Omega)$ for some $\beta\in (0,1)$, then $u_{\Omega,f}|_{\Omega}\in C_{loc}^{2s+\beta}(\Omega)$.
 \end{lemma}

 \begin{proof} Since the space $\mathring{H}^{s}(\R^n)$ is continuously embedded into $L^{\frac{2n}{n-2s}}(\R^n)$, we conclude by duality that $
 L^{\frac{2n}{n+2s}}(\R^n)\subset \mathring{H}^{-s}(\R^n)$.
 So we have both $f$ and $-f$ belong to the space $\mathring{H}^{-s}(\R^n)$.

   According to the Riesz representation theorem, for any $-f\in \mathring{H}^{-s}(\R^n)$, there exists a unique potential $\varphi \in \mathring{H}^{s}(\R^{n})$ such that
 \begin{equation}\label{d3}
 (\varphi, v)_{\mathring{H}^{s}(\R^{n})}=(-f,v)\ \ \ \ \forall \ v\in \mathring{H}^{s}(\R^{n}),
 \end{equation}
 where
  $(\varphi,v)_{\mathring{H}^{s}(\R^{n})}= \int_{\R^{n}}\int_{\R^{n}} \frac{(\varphi(x)-\varphi(y))(v(x)-v(y))}{|x-y|^{n+2s}} dxdy$
  and
  $(-f,\cdot):\mathring{H}^{s}(\R^{n})\rightarrow \R$ denotes the bounded linear functional generated by $-f$. Moreover,
 \begin{equation*}
 \|\varphi\|_{\mathring{H}^{s}(\R^{n})}= \|-f\|_{\mathring{H}^{-s}(\R^{n})}.
 \end{equation*}
 The function $\varphi \in \mathring{H}^{s}(\R^{n})$ satisfying \eqref{d3} is interpreted as the weak solution of
 \begin{equation*}
 (-\Delta)^{s}\varphi=-f\ \ \ \ \text{in} \ \R^{n},
 \end{equation*}
 see also \cite{LMM2015}.

 Next, we show that $\varphi\in L^{\infty}(\R^n)$.
By Lemma 1.8 of \cite{MH1972},
\begin{equation*}
(-\Delta)^{s}(\int_{\R^n}\frac{-f(y)}{|x-y|^{n-2s}}dy)=-f(x),
\end{equation*}
and by the uniqueness of $\varphi$, then
\begin{equation*}
\varphi(x)= \int_{\R^n}\frac{-f(y)}{|x-y|^{n-2s}}dy.
\end{equation*}
Let $\R^{n}=B_{1}(x)\cup (\R^{n}\setminus B_{1}(x))$, since $f\in L^{\infty}(\R^{n})\cap L^{\frac{2n}{n+2s}}(\R^{n})$, then we have

\begin{equation*} \begin{aligned}
  |\varphi(x)|&\leq\int_{B_{1}(x)}\frac{|f(y)|}{|x-y|^{n-2s}}dy+\int_{\R^{n}\setminus B_{1}(x)}\frac{|f(y)|}{|x-y|^{n-2s}}dy\\
  &\leq \|f\|_{L^{\infty}(B_{1}(x))}\int_{B_{1}(x)}\frac{1}{|x-y|^{n-2s}}dy\\
  &+\|f\|_{L^{\frac{2n}{n+2s}}(\R^{n}\setminus B_{1}(x)}\left (\int_{\R^{n}\setminus B_{1}(x)}\frac{1}{|x-y|^{2n}}dy\right)^{\frac{n-2s}{2n}}\\
  &\leq \|f\|_{L^{\infty}(\R^{n})}\int_{B_{1}(0)}\frac{1}{|y|^{n-2s}}dy+\|f\|_{L^{\frac{2n}{n+2s}}(\R^{n})}\left(\int_{\R^{n}\setminus B_{1}(0)}\frac{1}{|y|^{2n}}dy\right )^{\frac{n-2s}{2n}}\\
  &<\infty.
  \end{aligned}
  \end{equation*}
Thus, we have $\varphi\in L^{\infty}(\R^n)$.

Since for any $v\in \mathring{H}^{s}(\R^n)$,
 \begin{equation*}
 \frac{1}{2}\|v\|_{\mathring{H}^{s}(\R^n)}^{2}-\int_{\R^{n}}v(x)f(x)dx=\frac{1}{2}\|v+\varphi\|_{\mathring{H}^{s}(\R^n)}^{2}-\frac{1}{2}\|\varphi\|_{\mathring{H}^{s}(\R^n)}^{2}.
 \end{equation*}
 Then the function $\omega_{\Omega,f}:=u_{\Omega,f}+\varphi$ is the minimizer of the following problem
 \begin{equation}\label{d4}
 \text{argmin}\{\|\omega\|_{\mathring{H}^{s}(\R^n)}^{2}:\omega\in \mathring{H}^{s}(\R^n),\ \omega|_{\Omega^{c}}=\varphi\}.
 \end{equation}

We define the function $\bar{\omega}$ by
\begin{equation}\label{aa27}
\bar{\omega}:=\text{min}(\text{max}(\omega,-\|\varphi\|_{L^{\infty}(\R^{n})}),\|\varphi\|_{L^{\infty}(\R^{n})}).
\end{equation}

If $\omega\in \mathring{H}^{s}(\R^n)$ and $\omega|_{\Omega^{c}}=\varphi$, then by the definition of $\bar{\omega}$, we have $\bar{\omega}|_{\Omega^{c}}=\varphi$ and
$\|\bar{\omega}\|_{\mathring{H}^{s}(\R^n)}^{2}\leq \|\omega\|_{\mathring{H}^{s}(\R^n)}^{2}$. Thus, $\|\bar{\omega_{\Omega,f}}\|_{\mathring{H}^{s}(\R^n)}^{2}\leq \|\omega_{\Omega,f}\|_{\mathring{H}^{s}(\R^n)}^{2}$
Since $\omega_{\Omega,f}$ is the unique minimizer of the problem \eqref{d4}, then $\|\bar{\omega_{\Omega,f}}\|_{\mathring{H}^{s}(\R^n)}^{2}= \|\omega_{\Omega,f}\|_{\mathring{H}^{s}(\R^n)}^{2}$ and $\bar{\omega_{\Omega,f}}=\omega_{\Omega,f}$. Then  $\|\omega_{\Omega,f}\|_{L^{\infty}(\R^n)}=\|\bar{\omega_{\Omega,f}}\|_{L^{\infty}(\R^n)}\leq \|\varphi\|_{L^{\infty}(\R^n)}$. Thus,
 \begin{equation*}
 \|u_{\Omega,f}\|_{L^{\infty}(\R^n)}\leq \|\omega_{\Omega,f}\|_{L^{\infty}(R^n)}+\|\varphi\|_{L^{\infty}(\R^n)}\leq 2\|\varphi\|_{L^{\infty}(\R^n)}.
 \end{equation*}
 In addition, if $f\in C_{loc}^{\beta}(\Omega)$, then $u_{\Omega,f}\in C_{loc}^{2s+\beta}$ is a
 consequence of  \cite{RX2016} (see also \cite{L1972,RS2014}) whenever $2s+\beta$ is an integer or not.

 \end{proof}

 Before establishing the variational formula, we recall some definitions and basic facts. We assume that $\Omega$ is an open set not necessarily bounded, $\partial \Omega \in C^{2}$ is compact and bounded.

Let $X\in C^{\infty}(\R^{n},\R^{n})$ be a smooth vector filed. We consider a family of deformations $\{\Phi_{t}\}_{t \in (-1,1)}$ with the following properties:
\begin{equation}\label{b28}
\begin{aligned}
&\Phi_{t} \in C^{2}(\R^n;\R^n)\ \mathrm{for}\ t \in (-1,1),\\
 &\Phi_0=\mathrm{Id}_{\R^n}, \mathrm{and}\ \Phi_{t}=\text{Id} +tX.
\end{aligned}
\end{equation}
It is well-known that
\begin{equation}\label{aa28}
\mathrm{Jac} \Phi_{t}(x) =1+t\mathrm{div} X.
\end{equation}

Here, we choose the vector field $X(x)$ on $\R^{n}$ such that
\begin{equation*}
X|_{\partial \Omega}(x)=h_{L}(\nu_{\Omega}(x))\nu_{\Omega}(x),
\end{equation*}
 then $\Phi_{t}(\Omega)=(I+tX)(\Omega)=\Omega+tL$, see also \cite{DS2013}.

 The $s$-normal derivative of a function $u$ at $x \in \partial\Omega$ is defined by
\begin{equation}\label{d2}
\partial_{\nu}^{s}u(x):=\lim\limits_{t\rightarrow0}\frac{u(x+t\nu(x))-u(x)}{t^{s}},
\end{equation}
where $\nu(x)$ is the outerward normal vector to $\partial \Omega$ at $x$. When the function $u$ vanishes at the points of $\Omega^{c}$, we define
\begin{equation*}
\partial_{\nu}^{s}u(x):=-\lim\limits_{t\rightarrow0}\frac{u(x-t\nu(x))}{t^{s}},
\end{equation*}
In particular, for $s=1$, \eqref{d2} coincides with the classical normal derivative.
Let $\delta(x)=dist(x,\R^{n}\backslash\Omega)-dist(x,\Omega)$ be the signed distance function. Since we assume $\partial\Omega\in C^{2}$, then the signed distance function $\delta$ is also $C^{2}$ near the $\partial \Omega$ but not globally on $\R^{n}$.
Since $\partial \Omega$ is bounded, then there exists a Ball $B_{R}(0)\subset \R^{n}$ such that $\partial \Omega \subset B_{R}(0)$.
We assume that $\tilde{\Omega}=\Omega \cap B_{R}(0)$. For $x\in \partial\Omega$, the function
\begin{equation}\label{f1}
\frac{u}{\delta^{s}}(x):=\lim\limits_{y\rightarrow x, y\in \Omega}\frac{u}{\delta^{s}}(y),
\end{equation}
is well defined on $\partial \Omega$ as a limit, since the function $\frac{u}{\delta^{s}}$  extends to a function in $C^{\beta}(\tilde{\Omega})$ for some $\beta>0$, see \cite{RS2014}.
We define $\psi(x):=\frac{u}{\delta^{s}}(x)$, then $
\psi(x)=\lim\limits_{y\rightarrow x, y\in \Omega}\psi(y)$.
Additionally, the function $\delta^{1-s}\nabla u \cdot \nu=s\frac{u}{\delta^{s}}$, see \cite{FJ2021}. Since we consider $u|_{\partial \Omega}=0$, then the functional $\frac{u}{\delta^{s}}$ plays the role of an fractional inner normal derivative  for $x\in \partial\Omega$ and
\begin{equation*}
|\frac{u}{\delta^{s}}(x)|=|\partial_{\nu}^{s}u(x)|.
\end{equation*}

Here we introduce some critical mappings in \cite{DF2021}.

For $\varepsilon>0$, let $\Omega^{\varepsilon}=\{x\in \R^{n}:|\delta(x)|\leq \varepsilon\}$. We define the map
\begin{equation}\label{c1}
\begin{aligned}
\Psi&: \partial\Omega \times (-\varepsilon,\varepsilon)\rightarrow \Omega^{\varepsilon}\\
&\Psi(\sigma,r)=\sigma+r\nu(\sigma),
\end{aligned}
\end{equation}
where $\nu$ is the unit outer normal at $\sigma$.
Since  $\partial\Omega \in C^{2}$, then $\Psi$ is differentiable and $\delta(\Psi(\sigma,r))=r$ for $\sigma \in \partial\Omega$.
 Since we also assume $\partial \Omega$ is compact, then there exists an open ball $ B\subset \R^{n-1}$ centered at the origin and a  parametrization $f_{\sigma}$ with the property that for $\sigma \in \partial\Omega$,
\begin{equation*}
\begin{aligned}
&f_{\sigma}: B\rightarrow \partial \Omega, \ \ \ \ \ \ f_{\sigma}(0)=\sigma\\
&df_{\sigma}(0):\R^{n-1}\rightarrow \R^{n}.
\end{aligned}
\end{equation*}
For $z\in B$, we have
\begin{equation*}
f_{\sigma}(z)-f_{\sigma}(0)=df_{\sigma}(0)z+O(|z|^{2}).
\end{equation*}
Let $\nu_{\sigma}(z):=\nu(f_{\sigma}(z))$ for $z\in B$. We define
\begin{equation}\label{c2}
\begin{aligned}
&\Psi_{\sigma}:(-\varepsilon,\varepsilon)\times B\rightarrow  \Omega^{\varepsilon} \\
&\Psi_{\sigma}(r,z)=f_{\sigma}(z)+r\nu_{\sigma}(z).
\end{aligned}
\end{equation}
Then $\Psi_{\sigma}$ is a bi-Lipschitz map that maps $(-\varepsilon,\varepsilon)\times B$ onto a neighborhood of $\sigma$. In particular, we have $\Psi_{\sigma}(r,0)=\Psi(\sigma,r)$.
Further details about these important mappings have been explored, and can be found in \cite{DF2021}.

The next lemma shows the uniformly continuity of $\psi_{t}:=\frac{u_{t}}{\delta_{t}^{s}}$, which is  crucial for the computation of the Hadamard variational formula of $J_{f}(\Omega)$.
\begin{lemma}\label{A7}
Let $\Omega, \Omega_{t}\in C^{1,\beta}$ for some $\beta\in (0,1)$ be unbounded open sets and $\partial \Omega, \partial\Omega_{t}$ be compact and uniformly bounded. Let $f\in L^{\infty}(\R^{n})\cap L^{\frac{2n}{n+2s}}(\R^{n})$. Suppose that $u$ and $u_{t}$ satisfy the problem \eqref{d1} and $\Omega_{t}\rightarrow \Omega$ in the Hausdorff distance, as $t\rightarrow 0$. Then $u_{t}\rightarrow u$ a.e. in $\R^{n}$. Moreover, for $x_{t}\in \partial\Omega_{t}$ and $x\in \partial\Omega$, if $x_{t}\rightarrow x$, then $\psi_{t}(x_{t}):=\frac{u_{t}}{\delta_{t}^{s}}(x_{t})\rightarrow \psi(x):=\frac{u}{\delta^{s}}(x)$, as $t\rightarrow 0$, where $\delta_{t}(x)=dist(x,\R^{n} \backslash \Omega_{t})-dist(x,\Omega_{t})$.
\end{lemma}

\begin{proof}
Since $\partial \Omega_{t}$ is compact and uniformly bounded, we may choose $R_{1}>2R_{0}>0$ such that $\partial \Omega_{t}\subset B_{R_{0}/2}(0)\subset \R^{n}$ and $B_{R_{0}}(0)\subset B_{R_{1}/2}(x_{0})$ for all $t>0$ and some $x_{0}\in \partial\Omega_{t}$. Denote $\tilde{\Omega_{t}}=\Omega_{t}\cap B_{R_{0}+R_{1}}(0)$, then $\tilde{\Omega_{t}}$ is a bounded set in $\R^{n}$.

By Lemma \ref{A4}, we know that there exists a constant $C_{0}>0$ independent $t$  such that
\begin{equation*}
\|u_{t}\|_{L^{\infty}(\R^{n})}\leq C_{0}.
\end{equation*}
By Proposition 1.1 in \cite{RS2017}, $u_{t}$ is uniformly bounded in $C^{s}(B_{R_{0}}(0))$, that is, there exists a constant $C_{1}>0$ depending only on the $C^{1,\beta}$ norm of boundary of $\partial\Omega_{t}$. Since $\partial \Omega_{t}$ is uniformly bounded and uniformly of class $C^{1,\beta}$, then
\begin{equation*}
\|u_{t}\|_{C^{s}(B_{R_{0}}(0))}\leq C_{1}.
\end{equation*}
 According to the Arzel\`a-Ascoli theorem, the functions $u_{t}$ (up to a subsequence) converge to $u^{*}$ uniformly in $\bar{B}_{R_{0}}$.

Next, we will show that $u^{*}=u$ a.e. in $B_{R_{0}}(0)$. We first establish the $\Gamma-$convergence of the functional $I_{\Omega_{t},f}$ to $I_{\Omega,f}$ with respect to the weak convergence in $\mathring{H}^{s}(\R^n)$.
The latter is a natural topology, since the minimizers of $\Omega_{t,f}$ are uniformly bounded in $\mathring{H}^{s}(\R^n)$ independently of $t$.
The $\Gamma$-liminf follows from lower-semicontinuity of the $\mathring{H}^{s}(\R^n)$-norm and the linear term, together with the fact that the limit function vanishes a.e. in
$\Omega^{c}$ by the compact embedding of $\mathring{H}^{s}(\R^n)$ into $L_{loc}^{p}(\R^{n})$ for any $p<\frac{2n}{n-2s}$.
On the other hands, the $\Gamma$-limsup follows by approximating the limit function by a function from $C_{c}^{\infty}(\Omega)$, for which we have pointwise convergence of $I_{\Omega_{t},f}$, and a diagonal argument.
Thus,
we have the minimizer $u_{t}\rightharpoonup u$ in $\mathring{H}^{s}(\R^n)$. Since the function $I_{\Omega,f}$ admits the unique minimizer, then $u_{t}\rightarrow u$ a.e. in $\R^{n}$.
In particular, $u^{*}=u$ a.e. in $B_{\R_{0}}(0)$.

We consider the function $\psi_{t}:=\frac{u_{t}}{\delta_{t}^{s}}: \tilde{\Omega_{t}}\rightarrow \R$. By Theorem 1.2 in \cite{RS2017}, then $\psi_{t}$ is uniformly bounded in $C^{\beta}(\bar{B}_{R_{0}}(0))$. By the classical extension theorem (Theorem 6.38) in \cite{GT1983}, we can extend $\psi_{t}$ to a function $\bar{\psi}_{t}:\bar{B}_{R_{0}}(0)\rightarrow \R$ such that
\begin{equation*}
\|\bar{\psi_{t}}\|_{C^{\beta}(\bar{B}_{R_{0}}(0))}\leq C_{1}\|\psi_{t}\|_{C^{\beta}(\bar{B}_{R_{0}}(0))}\leq C,
\end{equation*}
where $C_{1}$ and $C$ are independent of $t$. According to the Arzel\`a-Ascoli theorem, there exists a function $\bar{\psi}^{*}\in C^{\alpha}(\bar{B}_{R_{0}}(0))$ such that $\bar{\psi_{t}}$ (up to a subsequence) converges to $\bar{\psi}^{*}$ uniformly. Since $u_{t}\rightarrow u$ and $\delta_{t}\rightarrow \delta$, then $\bar{\psi}^{*}$ is the extension of $\psi$. Moreover, we have $\bar{\psi}^{*}|_{\tilde{\Omega}\cap \bar{B}_{R_{0}(0)}}=\psi$. Finally, $\bar{\psi_{t}}$ is a continuous extension of $\partial_{\nu}^{s}u_{t}$ since for $x\in \partial \Omega_{t}$,
\begin{equation*}
\bar{\psi_{t}}(x)=\lim\limits_{s\rightarrow 0+}\psi_{t}(x-s\nu(x))=\lim\limits_{s\rightarrow 0+}
\frac{u_{t}(x-s\nu(x))}{\delta_{t}(x-s\nu(x))^{s}}=-\partial_{\nu}^{s}u_{t}(x).
\end{equation*}
In particular, for $x_{t}\in \partial \Omega_{t}, x\in \partial \Omega$ and $x_{t}\rightarrow x$, then
\begin{equation*}
\begin{aligned}
&|\frac{u_{t}}{\delta_{t}^{s}}(x_{t})-\frac{u}{\delta^{s}}(x)|=|\bar{\psi}_{t}(x_{t})-\bar{\psi}(x)|\\
&\leq |\bar{\psi}_{t}(x_{t})-\bar{\psi}_{t}(x)|+|\bar{\psi}_{t}(x)-\bar{\psi}(x)|\\
&\rightarrow 0.
\end{aligned}
\end{equation*}
Therefore, $\frac{u_{t}}{\delta_{t}^{s}}(x_{t})\rightarrow \frac{u}{\delta^{s}}(x)$, as $t\rightarrow0$.

\end{proof}
In the following theorem, we obtain the first variational formula for $J_{f}(\Omega)$.
\begin{theorem}\label{A1}
 Let $\Omega$ be an  open set with compact boundary of class $C^{2}$. Let $f \in L^{\infty}(\R^{n})\cap L^{\frac{2n}{n+2s}}(\R^{n})\cap C_{loc}^{\beta}(\R^{n})$ for some $\beta\in (0,1)$. The functional $J_{f}(\Omega)$ is defined as \eqref{c3}. Then we have
\begin{equation*}
\frac{d}{dt}J_{f}(\Phi_t(\Omega))\Big|_{t=0}=-\frac{1}{2}c_{0}\int_{\partial \Omega}|\frac{u_{\Omega,f}}{\delta^{s}}(\sigma)|^{2}h_{L}(\nu(\sigma)) d\mathcal{H}^{n-1}(\sigma),
\end{equation*}
where $c_{0}=c_{s}c_{n,s}a_{n,s}$ is a positive constant and $\nu(\sigma)$ is the unit outward normal vector to $\partial\Omega$.
\end{theorem}

\begin{proof}Let $\Omega_{t}=\Phi_{t}(\Omega)$.
Since $\partial \Omega$ is of class $C^{2}$, for all $x\in \partial \Omega$, $y\in \partial \Omega_{t}$ and sufficiently small $t$, we can write
\begin{equation}\label{aaa1}
\begin{aligned}
&y=\Phi_{t}(x),\\
&x=\Phi_{t}^{-1}(y)=y+t\rho_{t}(y)\nu_{t}(y),
\end{aligned}
\end{equation}
where $\nu_{t}$ is the unit outward normal to $\partial \Omega_{t}$ and $\rho_{t}\in C^{2}(\partial \Omega_{t})$. Moreover, $\Phi_{t}^{-1}$ establishes a bijection between $\partial \Omega_{t}$ and $\partial\Omega$. Then for $x\in \partial \Omega$, we have
\begin{equation*}
\rho_{0}(x)=\lim\limits_{t\rightarrow 0}\rho_{t}(\Phi_{t}(x))=\lim\limits_{t\rightarrow 0}\frac{(x-\Phi_{t}(x))\cdot \nu_{t}(\Phi_{t}(x))}{t}=-h_{L}(\nu(x)).
\end{equation*}

For sufficiently small $t>0$, we consider $\Omega_{t}$ as a regular inward deformation of $\Omega$, characterized by satisfying \eqref{aaa1} with $\rho_{t}\geq 0$. Note that the focus here is on inward perturbations, as for outward perturbations, one would merely interchange the roles of $\Omega_{t}$  and $\Omega$.

We denote $u=u_{\Omega,f}$ and $u_{t}=u_{\Omega_t,f}$ for simplicity. According to Lemma \ref{A4}, both $u$ and $u_t$ solve pointwise the equations
\begin{equation}\label{g11}
\left\{
\begin{aligned}
&(-\Delta)^{s}u=f, \ \ \ \ \ \ \mathrm{in}\  \Omega,\\
 &u=0,\ \ \ \ \ \ \ \ \ \ \ \ \ \ \ \mathrm{in}\  \Omega^{c}.
\end{aligned}
\right.
\end{equation}
and
\begin{equation}\label{g12}
\left\{
\begin{aligned}
&(-\Delta)^{s}u_{t}=f, \ \ \ \ \ \ \mathrm{in}\  \Omega_t,\\
 &u_{t}=0,\ \ \ \ \ \ \ \ \ \ \ \ \ \ \ \mathrm{in}\  \Omega_t^{c}.
\end{aligned}
\right.
\end{equation}

From \eqref{g11},\eqref{g12} and \eqref{c3}, we have
\begin{equation}\label{g2}
\begin{aligned}
&\frac{J_{f}(\Omega_{t})-J_{f}(\Omega)}{t}\\
&=-\frac{1}{2}\int_{\R^{n}}\frac{u_{t}(x)(-\Delta)^{s}u_{t}(x)}{t}dx+\frac{1}{2}\int_{\R^{n}}\frac{u(x)(-\Delta)^{s}u(x)}{t}dx\\
&=-\frac{1}{2}\int_{\R^{n}}\frac{(u_{t}(x)+u(x))(-\Delta)^{s}(u_{t}(x)-u(x))}{t}dx\\
&=-\frac{1}{2}\int_{\Omega\backslash \Omega_{t}}\frac{u(x)(-\Delta)^{s}u_{t}(x)}{t}dx+\frac{1}{2}\int_{\Omega\backslash \Omega_{t}}\frac{u(x)f(x)}{t}dx.
\end{aligned}
\end{equation}

We begin by addressing the second term on the right hand side of \eqref{g2}.
By the theorem 1.2 of \cite{RS2017}, we have $|u(x)|\leq C |\delta(x,\partial \Omega)|^{s}\leq C|t\rho_{t}|^{s}$.
Since $f\in L^{\infty}(\R^{n})$ and $|\Omega\backslash \Omega_{t}|\leq Ct$, then
\begin{equation*}
\begin{aligned}
\Big{|}\int_{\Omega\backslash \Omega_{t}}u(x)f(x)dx\Big{|}\leq C t^{s+1}\|f\|_{L^{\infty}(\Omega\backslash \Omega_{t})}\|\rho_{t}\|_{L^{\infty}(\partial \Omega_{t})}^{s}.
\end{aligned}
\end{equation*}
Hence, we deduce that
\begin{equation*}
\lim\limits_{t\rightarrow 0}\frac{1}{2}\int_{\Omega\backslash \Omega_{t}}\frac{u(x)f(x)}{t}dx=0.
\end{equation*}
So we only need to estimate the first term on the right hand side of \eqref{g2}.

Together with \eqref{F-S-1} and  $u_{t}=0$ in $\Omega_{t}^{c}$, then
\begin{equation*}
\begin{aligned}
&-\frac{1}{2}\int_{\Omega\backslash \Omega_{t}}u(x)(-\Delta)^{s}u_{t}(x)dx\\
&=-\frac{1}{2}c_{n,s}\int_{\Omega\backslash \Omega_{t}}\int_{\R^{n}}\frac{u(x)(u_{t}(x)-u_{t}(y))}{|x-y|^{n+2s}}dydx\\
&=\frac{1}{2}c_{n,s}\int_{\Omega\backslash \Omega_{t}}\int_{\Omega_{t}}\frac{u(x)u_{t}(y)}{|x-y|^{n+2s}}dydx.
\end{aligned}
\end{equation*}
Recall the mapping $\Psi:\partial \Omega_{t}\times(-\varepsilon,\varepsilon)\rightarrow (\Omega_{t})^{\varepsilon}$ and consider a change of variable $x=\Psi(\sigma,r)$, then
\begin{equation*}
\begin{aligned}
&\frac{1}{2}c_{n,s}\int_{\Omega\backslash \Omega_{t}}\int_{\Omega_{t}}\frac{u(x)u_{t}(y)}{|x-y|^{n+2s}}dydx\\
&=\frac{1}{2}c_{n,s}\int_{\partial\Omega_{t}}\int_{0}^{t\rho_{t}(\sigma)}\int_{\Omega_{t}}\frac{u(\Psi(\sigma,r))u_{t}(y)}{|\Psi(\sigma,r)-y|^{n+2s}}Jac\Psi(\sigma,r)dydrd\mathcal{H}^{n-1}(\sigma)\\
\end{aligned}
\end{equation*}
For a fixed $\sigma \in \partial \Omega_{t}$, we divide $\Omega_{t}$ into two disjoint parts $\Psi_{\sigma}(0,\varepsilon)\times B$ and $\Omega_{t}\setminus [\Psi_{\sigma}(0,\varepsilon)\times B]$. We write
\begin{equation*}
\begin{aligned}
A(\sigma,r)&=\int_{\Psi_{\sigma}(0,\varepsilon)\times B}\frac{u(\Psi(\sigma,r))u_{t}(y)}{|\Psi(\sigma,r)-y|^{n+2s}}dy,\\
B(\sigma,r)&=\int_{\Omega_{t}\setminus [\Psi_{\sigma}(0,\varepsilon)\times B]}\frac{u(\Psi(\sigma,r))u_{t}(y)}{|\Psi(\sigma,r)-y|^{n+2s}}dy.
\end{aligned}
\end{equation*}
Then we have
\begin{equation*}
\begin{aligned}
&\frac{1}{2}c_{n,s}\int_{\Omega\backslash \Omega_{t}}\int_{\Omega_{t}}\frac{u(x)u_{t}(y)}{|x-y|^{n+2s}}dydx\\
&=\frac{1}{2}c_{n,s}\int_{\partial\Omega_{t}}\int_{0}^{t\rho_{t}(\sigma)}A(\sigma,r)Jac\Psi(\sigma,r)drd\mathcal{H}^{n-1}(\sigma)\\
&+
\frac{1}{2}c_{n,s}\int_{\partial\Omega_{t}}\int_{0}^{t\rho_{t}(\sigma)}B(\sigma,r)Jac\Psi(\sigma,r)drd\mathcal{H}^{n-1}(\sigma)\\
&=\frac{1}{2}c_{n,s}\int_{\partial\Omega_{t}}\int_{0}^{1}A(\sigma,t\rho_{t}r)Jac\Psi(\sigma,t\rho_{t}r)t\rho_{t}drd\mathcal{H}^{n-1}(\sigma)\\
&+
\frac{1}{2}c_{n,s}\int_{\partial\Omega_{t}}\int_{0}^{1}B(\sigma,t\rho_{t}r)Jac\Psi(\sigma,t\rho_{t}r)t\rho_{t}drd\mathcal{H}^{n-1}(\sigma)\\
&:=I+II.
\end{aligned}
\end{equation*}

To estimate $I$, we shall give an estimate for $A(\sigma,t\rho_{t}r)$. Recall the mapping $\Psi_{\sigma}:(-\varepsilon,\varepsilon)\times B\rightarrow \Omega^{\varepsilon}$ and consider a change of variables $y=\Psi_{\sigma}(-\tilde{r},z)$, we have
\begin{equation*}
\begin{aligned}
&A(\sigma,t\rho_{t}r)\\
&=\int_{\Psi_{\sigma}(0,\varepsilon)\times B}\frac{u(\Psi(\sigma,t\rho_{t}r))u_{t}(y)}{|\Psi(\sigma,t\rho_{t}r)-y|^{n+2s}}dy\\
&=\int_{0}^{\varepsilon}\int_{B}\frac{u(\Psi(\sigma,t\rho_{t}r))u_{t}(\Psi_{\sigma}(-\tilde{r},z))}{|\Psi(\sigma,t\rho_{t}r)-\Psi_{\sigma}(-\tilde{r},z)|^{n+2s}}Jac\Psi_{\sigma}(-\tilde{r},z)dzd\tilde{r}\\
&=t\rho_{t}\int_{r}^{\frac{\varepsilon}{t\rho_{t}}+r}\int_{B}\frac{u(\Psi(\sigma,t\rho_{t}r))u_{t}(\Psi_{\sigma}(t\rho_{t}(-\tilde{r}+r),z))}{|\Psi(\sigma,t\rho_{t}r)-\Psi_{\sigma}(t\rho_{t}(-\tilde{r}+r),z)|^{n+2s}}Jac\Psi_{\sigma}(t\rho_{t}(-\tilde{r}+r),z)dzd\tilde{r}\\
&=(t\rho_{t})^{n}\int_{r}^{\frac{\varepsilon}{t\rho_{t}}+r}\int_{\frac{B}{t\rho_{t}\tilde{r}}}\frac{u(\Psi(\sigma,t\rho_{t}r))u_{t}(\Psi_{\sigma}(t\rho_{t}(-\tilde{r}+r),t\rho_{t}\tilde{r}z))}{|\Psi(\sigma,t\rho_{t}r)-\Psi_{\sigma}(t\rho_{t}(-\tilde{r}+r),t\rho_{t}\tilde{r})|^{n+2s}}\\
&\cdot Jac\Psi_{\sigma}(t\rho_{t}(-\tilde{r}+r),t\rho_{t}\tilde{r}z)(\tilde{r})^{n-1}dzd\tilde{r}
\end{aligned}
\end{equation*}

For any $x\in \Omega\backslash \Omega_{t}$, $x=\Psi(\sigma,t\rho_{t}r)=\sigma+t\rho_{t}r\nu_{t}(\sigma)=x_{0}-(1-r)t\rho_{t}\nu(x_{0})$, where $\sigma\in \partial\Omega_{t}$, $x_{0}\in \partial\Omega$ and $r\in [0,1]$ then $\delta (x)=(1-r)t\rho_{t}$. Since $u=\psi\delta^{s}$ and $\psi_{t}\rightarrow \psi$ uniformly as $t\rightarrow 0$, then by lemma \ref{A7},
\begin{equation}\label{g3}
\begin{aligned}
u(\Psi(\sigma,t\rho_{t}r))&=\psi(\Psi(\sigma,rt\rho_{t}))(t\rho_{t}(1-r))^{s}\\
&=(1+o(t))\psi_{t}(\sigma)(t\rho_{t}(1-r))^{s}
\end{aligned}
\end{equation}
Similarly, $u_{t}(x)=\psi_{t}(x)\delta(x,\partial \Omega_{t})^{s}$ for $x\in \Omega_{t}$, then
\begin{equation}\label{g4}
\begin{aligned}
u_{t}(\Psi_{\sigma}(t\rho_{t}(\tilde{r}+r),t\rho_{t}\tilde{r}z))&=\psi_{t}(\Psi_{\sigma}(t\rho_{t}(\tilde{r}+r),t\rho_{t}\tilde{r}z))|-t\rho_{t}\tilde{r}+t\rho_{t}r|^{s}\\
&=(1+o(t))\psi_{t}(\sigma)|-t\rho_{t}\tilde{r}+t\rho_{t}r|^{s}.
\end{aligned}
\end{equation}

It remains to consider the kernel differences $|\Psi(\sigma,t\rho_{t}r)-\Psi_{\sigma}(t\rho_{t}(-\tilde{r}+r),t\rho_{t}\tilde{r})|$.
By  $(v)$ of the Lemma 6.6 in \cite{DF2021}, there exists a positive constant $C_{0}$ such that
\begin{equation*}
\begin{aligned}
&\Big{|}\frac{|\Psi(\sigma,t\rho_{t}r)-\Psi_{\sigma}(t\rho_{t}(-\tilde{r}+r),t\rho_{t}\tilde{r}z)|^{2}}{|t\rho_{t}\tilde{r}|^{2}}-(1+|z|^{2})\Big{|}\\
&=\Big{|}\frac{|\Psi_{\sigma}(t\rho_{t}r,0)-\Psi_{\sigma}(t\rho_{t}(-\tilde{r}+r),t\rho_{t}\tilde{r}z)|^{2}}{|t\rho_{t}\tilde{r}|^{2}}-(1+|z|^{2})\Big{|}\\
&\leq C_{0}(|t\rho_{t}\tilde{r}|+|t\rho_{t}r|+|t\rho_{t}\tilde{r}z|)|z|^{2}.
\end{aligned}
\end{equation*}
Thus, we have
\begin{equation}\label{g10}
\lim\limits_{t\rightarrow 0}\frac{|\Psi(\sigma,t\rho_{t}r)-\Psi_{\sigma}(t\rho_{t}(-\tilde{r}+r),t\rho_{t}\tilde{r}z)|^{2}}{|t\rho_{t}\tilde{r}|^{2}}=1+|z|^{2}.
\end{equation}
By $(iv)$ and $(i)$ of Lemma 6.6 in \cite{DF2021}, then
\begin{equation}\label{g7}
\begin{aligned}
&|\Psi(\sigma,t\rho_{t}r)-\Psi_{\sigma}(t\rho_{t}(-\tilde{r}+r),t\rho_{t}\tilde{r})|\\
&=|\Psi_{\sigma}(t\rho_{t}r,0)-\Psi_{\sigma}(t\rho_{t}(-\tilde{r}+r),t\rho_{t}\tilde{r}z)|\\
&\geq \frac{1}{C_{0}}(|t\rho_{t}\tilde{r}|^{2}+(t\rho_{t}\tilde{r}z)^{2})^{1/2}=\frac{1}{C_{0}}|t\rho_{t}\tilde{r}|(1+|z|^{2})^{1/2},
\end{aligned}
\end{equation}
 and
\begin{equation}\label{g8}
\begin{aligned}
&|Jac\Psi_{\sigma}(t\rho_{t}(-\tilde{r}+r),t\rho_{t}\tilde{r}z)|\leq C_{0},\\
&|Jac\Psi(\sigma,t\rho_{t}r)|\leq C_{0}.
\end{aligned}
\end{equation}
Together with \eqref{g3},\eqref{g4},\eqref{g7} and \eqref{g8}, we have
\begin{equation*}
\begin{aligned}
&|A(\sigma,t\rho_{t}r)|\\
&\leq C_{0}^{2} |(1+o(t))\psi_{t}|^{2}\int_{r}^{\frac{\varepsilon}{t\rho_{t}}+r}\int_{\frac{B}{t\rho_{t}\tilde{r}}}
\frac{(1-r)^{s}|-\tilde{r}+r|^{s}}{\tilde{r}^{2s+1}}\frac{1}{(1+|z|^{2})^{\frac{n+2s}{2}}}dzd\tilde{r}\\
&\leq C_{0}^{2}|(1+o(t))\psi_{t}|^{2}\int_{r}^{\infty}\int_{\R^{n-1}}
\frac{(1-r)^{s}|-\tilde{r}+r|^{s}}{\tilde{r}^{2s+1}}\frac{1}{(1+|z|^{2})^{\frac{n+2s}{2}}}dzd\tilde{r}\\
&=C_{0}^{2}|(1+o(t))\psi_{t}|^{2}a_{n,s}\int_{r}^{\infty}\frac{(1-r)^{s}|-\tilde{r}+r|^{s}}{\tilde{r}^{2s+1}}d\tilde{r},
\end{aligned}
\end{equation*}
where $a_{n,s}=\int_{\R^{n-1}}\frac{1}{(1+|z|^{2})^{\frac{n+2s}{2}}}dz<\infty.$

By using the inequality $(1+x)^{s}\leq 1+sx$ for $x>-1$ and through a basic integral calculation, we obtain
\begin{equation*}
\begin{aligned}
c_{s}:&=\int_{0}^{1}\int_{r}^{\infty}\frac{(1-r)^{s}|-\tilde{r}+r|^{s}}{\tilde{r}^{2s+1}}d\tilde{r}dr\leq \int_{0}^{1}\int_{r}^{\infty}\frac{(\tilde{r}-r)^{s}}{\tilde{r}^{2s+1}}d\tilde{r}dr\\
&=\int_{0}^{1}\int_{r}^{\infty}\frac{\tilde{r}^{s}(1-\frac{r}{\tilde{r}})^{s}}{\tilde{r}^{2s+1}}d\tilde{r}dr
\leq\int_{0}^{1}\int_{r}^{\infty}\frac{\tilde{r}^{s}(1-s\frac{r}{\tilde{r}})}{\tilde{r}^{2s+1}}d\tilde{r}dr
\\&=\int_{0}^{1}\int_{r}^{\infty}\frac{1}{\tilde{r}^{s+1}}d\tilde{r}dr+
\int_{0}^{1}\int_{r}^{\infty}\frac{-sr}{\tilde{r}^{s+2}}d\tilde{r}dr
 <\infty.
\end{aligned}
\end{equation*}
Then we have
\begin{equation}\label{g13}
\frac{\int_{0}^{1}|A(\sigma,t\rho_{t}r)Jac\Psi(\sigma,t\rho_{t}r)t\rho_{t}|dr}{t}|\leq
C_{0}^{3}a_{n,s}c_{s}|(1+o(t))\psi_{t}|^{2}|\rho_{t}|.
\end{equation}

By the uniform continuity of $\rho_{t}$ and $\psi_{t}$, then we have
\begin{equation}\label{g14}
\begin{aligned}
&\lim\limits_{t\rightarrow0}C_{0}^{3}a_{n,s}c_{s}\int_{\partial \Omega_{t}}|(1+o(t))\psi_{t}|^{2}|\rho_{t}|d\mathcal{H}^{n-1}(\sigma)\\
&=C_{0}^{3}a_{n,s}c_{s}\int_{\partial\Omega}\psi^{2}(\sigma)h_{L}(\nu(\sigma))d\mathcal{H}^{n-1}(\sigma).
\end{aligned}
\end{equation}

Together with \eqref{g13},\eqref{g14} and \eqref{g10}, according to the generalized dominated convergence theorem, we have
\begin{equation}\label{g6}
\begin{aligned}
&\lim\limits_{t\rightarrow 0}\frac{\frac{1}{2}c_{n,s}\int_{\partial\Omega_{t}}\int_{0}^{1}A(\sigma,t\rho_{t}r)Jac\Psi(\sigma,t\rho_{t}r)t\rho_{t}drd\mathcal{H}^{n-1}(\sigma)}{t}\\
&=-\frac{1}{2}c_{n,s}c_{s}a_{n,s}\int_{\partial \Omega}(\frac{u}{\delta^{s}}(\sigma))^{2}h_{L}(\nu(\sigma))d\mathcal{H}^{n-1}(\sigma)
\end{aligned}
\end{equation}

Next we estimate $II$.
There exists $\varepsilon'\in (0,\frac{\varepsilon}{2})$ such that
\begin{equation}\label{g1}
|\sigma-y|\geq 3\varepsilon' \ \ \ \ \ \ \ \ \
\text{for}\ \text{all}\ y\in \ \Omega_{t}\backslash \Psi_{\sigma}[(0,\varepsilon)\times B].
\end{equation}
Moreover, $\varepsilon'$ can be chosen independently of $\sigma\in \partial \Omega_{t}$.

When $t$ is sufficiently small, $|t\rho_{t}r|<\varepsilon'$. By \eqref{g1}, we have
\begin{equation*}
\begin{aligned}
|\Psi(\sigma,t\rho_{t}r)-y|&=|\sigma+t\rho_{t}r\nu_{t}(\sigma)-y|\geq |\sigma-y|-|t\rho_{t}r|\\
&\geq\frac{|\sigma-y|}{3}+\frac{2|\sigma-y|}{3}-\varepsilon'\\
&\geq\frac{|\sigma-y|}{3}+\varepsilon'.
\end{aligned}
\end{equation*}

Since $|u|\leq C |t \rho_{t}r|^{s}\leq C |t\rho_{t}|^{s}$ for $r\in [0,1]$ and $u_{t}\in L^{\infty}(\R^{n})$, we have
\begin{equation*}
\begin{aligned}
|B(\sigma,t\rho_{t}r)|
&\leq C\|\rho_{t}\|_{L^{\infty}(\partial\Omega_{t})}|t|^{s}\|u_{t}\|_{L^{\infty}(\R^{n})}\int_{\R^{n}}\frac{1}{|\frac{|\sigma-y|}{3}+\varepsilon'|^{n+2s}}dy\\
&\leq C|t|^{s}.
\end{aligned}
\end{equation*}
Thus,
\begin{equation*}
\begin{aligned}
|\int_{\partial\Omega_{t}}\int_{0}^{1}B(\sigma,t\rho_{t}r)Jac\Psi(\sigma,t\rho_{t}r)t\rho_{t}drd\mathcal{H}^{n-1}(\sigma)|\leq C |t|^{s+1},
\end{aligned}
\end{equation*}
where $C$ is a  distinct positive constant.

Moreover, we have
\begin{equation}\label{g5}
\lim\limits_{t\rightarrow0}\frac{\frac{1}{2}c_{n,s}\int_{\partial\Omega_{t}}\int_{0}^{1}B(\sigma,t\rho_{t}r)Jac\Psi(\sigma,t\rho_{t}r)t\rho_{t}drd\mathcal{H}^{n-1}(\sigma)\\
}{t}=0.
\end{equation}
Together with  \eqref{g6} and \eqref{g5}, we have
\begin{equation*}
\frac{d}{dt}J_{f}(\Phi_t(\Omega))\Big|_{t=0}=-\frac{1}{2}c_{s}c_{n,s}a_{n,s}\int_{\partial \Omega}|\frac{u_{\Omega,f}}{\delta^{s}}(\sigma)|^{2}h_{L}(\nu(\sigma)) d\mathcal{H}^{n-1}(\sigma).
\end{equation*}
\end{proof}

\subsection{The Hadamard variational formula for Riesz capacity}
In this subsection, we will derive the Hadamard variational formula for Riesz capacity.

\begin{theorem}\label{A6}
Let $\Omega$ and $L$ be two bounded convex bodies of class $C^{2}$ with support functions $h_{\Omega}$ and $h_{L}$ respectively.
Let $\Omega_{t}$ be the Wulff shape of  $h_{t}=h_{\Omega}+th_{L}$ and $\nu$ be the Gauss map of  $\Omega$. Suppose $u_{\Omega}$ is the minimizer of $Cap_{\alpha}(\Omega)$, then we have
\begin{equation*}
\frac{d}{dt} \mathrm{Cap}_{\alpha}(\Omega_{t})\Big{|}_{t=0} =c_{0}\int_{\partial \Omega}|\partial_{\nu}^{\alpha/2}u_{\Omega}(\sigma)|^{2}h_{L}(\nu_{\Omega}(\sigma))d\mathcal{H}^{n-1}(\sigma),
\end{equation*}
where $c_{0}=c_{\alpha/2}a_{n,\alpha/2}c_{n,\alpha/2}$ is a positive constant.
\end{theorem}

\begin{proof}
Let $\varphi$ be a fixed function satisfying $\varphi\in C_{c}^{\infty}(\R^n)$  and $\varphi=1$  in an open neighborhood of $\Omega$. We define
\begin{equation*}
f(x)=-(-\Delta)^{s}\varphi(x),
\end{equation*}
then $f$ is well-defined and $f\in C^{\infty}(\R^n)\cap L^{\frac{2n}{n+2s}}(\R^{n}) \cap L^{\infty}(\R^{n})$, see section \ref{sec2}.  Assume that $u$ is a test function in the definition of $\text{Cap}_{\alpha}(\Omega)$ and $u=1$ on $\Omega$. Let $v=u-\varphi$, then $v$ satisfies the following equations

\begin{equation*}
\left\{
\begin{aligned}
& (-\Delta)^{s}v =f, \ \ \ \ \ \ \ \ \ \ \ \mathrm{in} \ \Omega^{c},  \\
&  v=0  , \ \ \ \ \ \ \ \ \ \ \ \ \ \ \ \ \ \ \ \ \mathrm{in} \ \Omega. \\
\end{aligned}
\right.
\end{equation*}
Then, we have
\begin{equation*}
\begin{aligned}
&\frac{1}{2}\|u\|_{\mathring{H}^{s}(\R^{n})}^{2}=\frac{1}{2}\|v+\varphi\|_{\mathring{H}^{s}(\R^{n})}^{2}\\
&=\frac{1}{2}\|v\|_{\mathring{H}^{s}(\R^{n})}^{2}+\frac{1}{2}\|\varphi\|_{\mathring{H}^{s}(\R^{n})}^{2}-\int_{\R^{n}}v(x)f(x)dx.
\end{aligned}
\end{equation*}
Furthermore, we obtain
\begin{equation*}
\frac{1}{2}\text{Cap}_{\alpha}(\Omega)=\frac{1}{2}\|\varphi\|_{\mathring{H}^{s}(\R^n)}^{2}+J_{f}(\Omega^{c}).
\end{equation*}
Let $u_\Omega$ be the minimizer of $\text{Cap}_{\alpha}(\Omega)$, then $u_{\Omega}=v_{\Omega^{c},f}+\varphi$, where $v_{\Omega^{c},f}$ is the minimizer of $I_{\Omega^{c},f}$. Since $\varphi=1$ in an open neighborhood of $\Omega$, then $\frac{\varphi}{\delta^{s}}(x)=0$ and $\frac{u_{\Omega}}{\delta^{s}}(x)=\frac{v_{\Omega^{c},f}}{\delta^{s}}(x)$ for all $x\in \partial \Omega$. Since $|\frac{u_{\Omega}}{\delta^{s}}(x)|=|\partial _{\nu}^{s}u(x)|$ for $x\in \partial\Omega$, it follows from theorem \ref{A1} that
\begin{equation*}
\begin{aligned}
\frac{d}{dt}  \mathrm{Cap}_{\alpha}(\Omega_{t})\Big{|}_{t=0}&=c_{0}\int_{\partial \Omega}|\frac{u_{\Omega}}{\delta^{s}}(\sigma)|^{2}h_{L}(\nu_{\Omega}(\sigma))d\mathcal{H}^{n-1}(\sigma)\\
&=c_{0}\int_{\partial \Omega}|\partial _{\nu}^{\alpha/2}u(\sigma)|^{2}h_{L}(\nu_{\Omega}(\sigma))d\mathcal{H}^{n-1}(\sigma)
\end{aligned}
\end{equation*}
where $c_{0}=c_{\alpha/2}a_{n,\alpha/2}c_{n,\alpha/2}$.
\end{proof}

\section{overdetermined problem for riesz capacity}\label{sec4}

In this section, we first provide a characterization of constrained local minimum for $\text{Cap}_{\alpha}$. We define a bounded open subset $\Omega$ as a constrained local minimum for $\text{Cap}_{\alpha}$ if, for all families of deformations $\Phi_{t}$ satisfying \eqref{b28} and the volume invariance condition $|\Phi_{t}(\Omega)|=|\Omega|$ for $t\in (-1,1)$, there exists $t_{0}\in (0,1)$ with $\text{Cap}_{\alpha}(\Phi_{t}(\Omega))\geq \text{Cap}_{\alpha}(\Omega)$ for $t\in (-t_{0},t_{0})$.

\begin{definition}
We say that $\Omega\in \mathcal{K}_{0}^{n}$ is stationary for a functional $F:\mathcal{K}_{o}^{n}\rightarrow\R^{+}$, if
\begin{equation*}
\frac{d}{dt}F((1-t)\Omega+tL)\big{|}_{t=0^{+}}=0,\ \ \text{for}\ \text{any} \ L\in \mathcal{K}_{o}^{n}.
\end{equation*}
\end{definition}

 Our result is as follows.
\begin{theorem}
If a bounded convex domain $\Omega \in \mathcal{K}_{o}^{n}$ with $C^{2}$ boundary is a volume constrained local minimum for $\Omega\rightarrow \text{Cap}_{\alpha}(\Omega)$, then $\Omega$ is a ball.
\end{theorem}

\begin{proof}
We define the functional
\begin{equation*}
L(\Omega)=\text{Cap}_{\alpha}(\Omega)+\lambda(V(\Omega)-c).
\end{equation*}
If $\Omega$ is a stationary point for $L$, then
\begin{equation*}
\begin{aligned}
\frac{d}{dt}L(\Omega+tL)\big{|}_{t=0^{+}}=\frac{d}{dt}\text{Cap}_{\alpha}(\Omega+tL)\big{|}_{t=0^{+}}+\lambda\frac{d}{dt}V(\Omega+tL)\big{|}_{t=0^{+}}=0.\\
\end{aligned}
\end{equation*}
From theorem \ref{A6}, we have
\begin{equation}\label{a28}
\begin{aligned}
&\frac{d}{dt}\text{Cap}_{\alpha}(\Omega+tL)\big{|}_{t=0^{+}}\\
&=c_{0}
\int_{\partial \Omega}|\partial_{\nu}^{\alpha/2}u(x)|^{2}h_{L}(\nu(x))d\mathcal{H}^{n-1}(x).
\end{aligned}
\end{equation}
It is also well-known that
\begin{equation*}
\begin{aligned}
&\frac{d}{dt}V(\Omega+tL)\big{|}_{t=0^{+}}\\
&=\int_{\partial \Omega}h_{L}(\nu(x))d\mathcal{H}^{n-1}(x).
\end{aligned}
\end{equation*}
Therefore, we obtain
\begin{equation*}
\begin{aligned}
\int_{\partial \Omega}h_{L}(\nu(x))\Big{[}|c_{0}\partial_{\nu}^{\alpha/2}u(x)|^{2}+\lambda \Big{]}d\mathcal{H}^{n-1}(x)=0.
\end{aligned}
\end{equation*}
Since $L$ is arbitrary, then
\begin{equation*}
|\partial_{\nu}^{\alpha/2}u(x)|^{2}=-\frac{\lambda}{c_{0}}.
\end{equation*}
By the theorem 1.1 of \cite{SV2019}, we obtain that $\Omega$ is a ball.
\end{proof}

Furthermore, we study another overdetermined problem for $\text{Cap}_{1}(\Omega)$ in a convex body $\Omega$.
  We replace the volume $V(\Omega)$ by mean width $\text{M}(\Omega)$. The mean width of $\Omega$ is defined by
\begin{equation*}
\text{M}(\Omega):=\frac{2}{\omega_{n}}\int_{S^{n-1}}h_{\Omega}(u)du=\frac{2}{\omega_{n}}\int_{\partial \Omega}h_{\Omega}(\nu(x))G(x)d\mathcal{H}^{n-1}(x),
\end{equation*}
where $G(x)$ is the Gaussian curvature of $\partial \Omega$ at $x$.

We introduce the extremum problem for $\text{Cap}_{1}(\Omega)$ with constraint of fixed mean width $\text{M}(\Omega)=c$.
We define the functional
\begin{equation*}
\tilde{L}(\Omega)=\text{Cap}_{1}(\Omega)+\lambda(M(\Omega)-c).
\end{equation*}
If $\Omega$ is a stationary point for $L$, then
\begin{equation*}
\begin{aligned}
\frac{d}{dt}\tilde{L}(\Omega+tL)\big{|}_{t=0^{+}}=\frac{d}{dt}\text{Cap}_{1}(\Omega+tL)\big{|}_{t=0^{+}}+\lambda\frac{d}{dt}M(\Omega+tL)\big{|}_{t=0^{+}}=0\\
\end{aligned}
\end{equation*}
It follows from Theorem \ref{A6}  that
\begin{equation}\label{a28}
\begin{aligned}
&\frac{d}{dt}\text{Cap}_{1}(\Omega+tL)\big{|}_{t=0^{+}}\\
&=c_{0}
\int_{\partial \Omega}|\partial_{\nu}^{1/2}u(x)|^{2}h_{L}(\nu(x))d\mathcal{H}^{n-1}(x).
\end{aligned}
\end{equation}
It is well-known that
\begin{equation*}
\begin{aligned}
&\frac{d}{dt}\text{M}(\Omega+tL)\big{|}_{t=0^{+}}\\
&=\frac{2}{\omega_{n}}\int_{\partial \Omega}h_{L}(\nu(x))G(x)d\mathcal{H}^{n-1}(x).
\end{aligned}
\end{equation*}
Therefore, we obtain
\begin{equation*}
\begin{aligned}
\int_{\partial \Omega}h_{L}(\nu(x))\Big{[}c_{0}|\partial_{\nu}^{1/2}u(x)|^{2}+\frac{2\lambda}{\omega_{n}}G(x)\Big{]}d\mathcal{H}^{n-1}(x)=0.
\end{aligned}
\end{equation*}
Since $L$ is arbitrary, then
\begin{equation*}
|\partial_{\nu}^{1/2}u(x)|^{2}=-\frac{2\lambda}{c_{0}\omega_{n}}G(x).
\end{equation*}
Thus, this gives the following overdetermined problem
\begin{equation}\label{a25}
\left\{
\begin{aligned}
&(-\Delta)^{\frac{1}{2}}u=0\ \ \ \ \ \ \ \ \ \ \ \ \ \ \ \ \ \ \ \ \ \ in \ \R^{n}\setminus \bar{\Omega},\\
&u=1  \ \ \ \ \ \ \ \ \ \ \ \ \ \ \ \ \ \ \ \ \ \ \ \ \ \ \ \ \ \ \ on  \ \partial \Omega,\\
&\lim\limits_{|x|\rightarrow \infty}u(x)=0,\\
& |\partial_{\nu}^{\frac{1}{2}}u(x)|^{2}=cG(x)\ \ \ \ \ \ \ \ \ \ \ \ \ \ on\ \partial \Omega.
\end{aligned}
\right.
\end{equation}

We now introduce the following definitions as posed in \cite{F2012}.
\begin{definition}
We say that $F: \mathcal{K}_{o}^{n}\rightarrow \R^{+}$ is a Brunn-Minkowski functional of degree $\alpha$, if\\
(i) F is rigid motion invariant:
\begin{equation*}
F(\phi(\Omega))=F(\Omega),\ \ \ \ \forall \Omega\in \mathcal{K}_{o}^{n},\ \forall \phi:\R^{n}\rightarrow \R^{n}\ \text{rigid} \ \text{motion};\\
\end{equation*}
(ii) F is Hausdorff continuous:
\begin{equation*}
F(\Omega_{n})\rightarrow F(\Omega),\ \ \ \ \text{whenever}\ \Omega_{n}\rightarrow \Omega\ \text{in}\ \text{Hausdorff}\ \text{distance};\\
\end{equation*}
(iii) F is Minkowski differentiable:
\begin{equation*}
\exists \frac{d}{dt}F((1-t)\Omega+tL)|_{t=0^{+}},\ \ \forall\ \Omega,L\in \mathcal{K}_{o}^{n};\\
\end{equation*}
(iv) F is $\alpha$-homogeneous for some $\alpha\neq 0$:
\begin{equation*}
F(t\Omega)=t^{\alpha}F(\Omega),\ \ \ \forall \Omega\in \mathcal{K}_{o}^{n},\forall t\in \R^{+};\\
\end{equation*}
(v) F satisfies the Brunn-Minkowski inequality:
\begin{equation*}
F^{1/\alpha}(\Omega+L)\geq F^{1/\alpha}(\Omega)+F^{1/\alpha}(L),\ \ \forall \Omega,L\in \mathcal{K}_{o}^{n},
\end{equation*}
with equality if and only if $\Omega$ and $L$ are homothetic.
\end{definition}
The Brunn-Minkowski inequality for 1-Riesz capacity was established for $n\geq 2$ in  \cite{NR2015}. Equality for $n = 2$ was demonstrated in \cite{MNR2018}, and its extension to higher-dimensional cases was also verified in \cite{QZ2023}.
Thus, $\text{Cap}_{1}(\Omega)$ is a Brunn-Minkowski functional of degree $n-1$.

The following lemma is important for characterizing the ball in \cite{F2012}.

\begin{lemma}\label{A5}
Let $F:\mathcal{K}_{o}^{n}\rightarrow \R^{+}$ be a Brunn-Minkowski functional of degree $\alpha$. If $K\in \mathcal{K}_{o}^{n}$ is a stationary domain for the functional $\frac{F(K)^{1/\alpha}}{M(K)}$, then $K$ is a ball.
\end{lemma}

Now we solve the following overdetermined problem.

\begin{theorem}
Let $\Omega\in \mathcal{K}_{o}^{n}$ be a bounded convex domain of class $C^{2}$, if there exists a solution $u$  to the following overdetermined boundary value problem
\begin{equation}\label{a25}
\left\{
\begin{aligned}
&(-\Delta)^{\frac{1}{2}}u=0\ \ \ \ \ \ \ \ \ \ \ \ \ \ \ \ \ \  \ \ in \ \ \ \R^{n}\setminus \bar{\Omega},\\
&u=1  \ \ \ \ \ \ \ \ \ \ \ \ \ \ \ \ \ \ \ \ \ \ \ \ \ \ \ \  \ on  \ \partial \Omega,\\
&\lim\limits_{|x|\rightarrow \infty}u(x)=0,\\
& |\partial_{\nu}^{\frac{1}{2}}u(x)|^{2}=cG(x)\ \ \ \ \ \ \ \ \ \ \ \ \ \ on\ \partial \Omega,
\end{aligned}
\right.
\end{equation}
then $\Omega$ is a ball.
\end{theorem}

\begin{proof}
Assume that $u$ is a solution of problem \eqref{a25}, then
\begin{equation}\label{aa29}
|\partial_{\nu}^{\frac{1}{2}}u(x)|^{2}=cG(x)\ \ \ \ \ on\ \partial \Omega.
\end{equation}

Multiplying both sides of equality \label{aa29} by $h_{\Omega}(\nu(x))$ and integrating over $\partial \Omega$, we have
\begin{equation*}
\int_{\partial \Omega}|\partial_{\nu}^{\frac{1}{2}}u(x)|^{2}h_{\Omega}(\nu(x))d\mathcal{H}^{n-1}(x)=
\int_{\partial \Omega}cG(x)h_{\Omega}(\nu(x))d\mathcal{H}^{n-1}(x).
\end{equation*}
By the definition of $\text{Cap}_{1}(\Omega)$ and $\text{M}(\Omega)$, then
\begin{equation}\label{a26}
c=\frac{2(n-1)\text{Cap}_{1}(\Omega)}{c_{0}\omega_{n}\text{M}(\Omega)}.
\end{equation}
Similarly, multiplying both sides of equality \label{aa29} by $h_{L}(\nu(x))-h_{\Omega}(\nu(x))$ and integrating over $\partial \Omega$, we have
\begin{equation}\label{a27}
\begin{aligned}
&\int_{\partial \Omega}|\partial_{\nu}^{\frac{1}{2}}u(x)|^{2}(h_{L}(\nu(x))-h_{\Omega}(\nu(x)))d\mathcal{H}^{n-1}(x)\\
&=\int_{\partial \Omega}cG(x)(h_{L}(\nu(x))-h_{\Omega}(\nu(x)))\mathcal{H}^{n-1}(x).
\end{aligned}
\end{equation}
By the Hadamard variational formula for $\text{Cap}_{1}(\Omega)$ and $\text{M}(\Omega)$, we obtain
\begin{equation}\label{a28}
\begin{aligned}
&\frac{d}{dt}\text{Cap}_{1}((1-t)\Omega+tL)^{\frac{1}{n-1}}\big{|}_{t=0^{+}}\\
&=\frac{1}{n-1}\text{Cap}_{1}(\Omega)^{\frac{2-n}{n-1}}c_{0}
\int_{\partial \Omega}|\partial_{\nu}^{1/2}u(x)|^{2}(h_{L}(\nu(x))-h_{\Omega}(\nu(x)))d\mathcal{H}^{n-1}(x),
\end{aligned}
\end{equation}
and
\begin{equation}\label{a29}
\begin{aligned}
&\frac{d}{dt}\text{M}((1-t)\Omega+tL)\big{|}_{t=0^{+}}\\
&=\frac{2}{\omega_{n}}\int_{\partial \Omega}(h_{L}(\nu(x))-h_{\Omega}(\nu(x)))G(x)d\mathcal{H}^{n-1}(x).
\end{aligned}
\end{equation}
It follows from \eqref{a26},\eqref{a27},\eqref{a28} and \eqref{a29} that
\begin{equation*}
\frac{d}{dt}\text{Cap}_{1}((1-t)\Omega+tL)^{\frac{1}{n-1}}\big{|}_{t=0^{+}}
=\frac{\text{Cap}_{1}(\Omega)^{\frac{1}{n-1}}}{M(\Omega)}\frac{d}{dt}\text{M}((1-t)\Omega+tL)\big{|}_{t=0^{+}},
\end{equation*}
which implies
\begin{equation*}
\frac{d}{dt}\frac{\text{Cap}_{1}((1-t)\Omega+tL)^{\frac{1}{n-1}}}{\text{M}((1-t)\Omega+tL)}\big{|}_{t=0^{+}}=0.
\end{equation*}
Thus, $\Omega$ is the stationary domain of the functional $\frac{\text{Cap}_{1}(\Omega)^{\frac{1}{n-1}}}{M(\Omega)}$. Since $\text{Cap}_{1}(\Omega)$ is a Brunn-Minkowski functional, then it follows from Lemma \ref{A5} that $\Omega$ is a ball.
\end{proof}

\section*{acknowledgement} The authors thank Prof.Yong Huang for his valued advices and comments on this subject.

\end{document}